\documentclass[final]{siamltex}
\usepackage{amsmath}
\usepackage{mathtools}
\usepackage{amssymb}
\usepackage{amsfonts}
\usepackage{amsxtra}
\usepackage{amstext}
\usepackage{amsbsy}
\usepackage{amscd}
\usepackage{graphicx}
\usepackage{float}
\usepackage{cite}
\usepackage{lmodern}
\usepackage{xcolor}
\usepackage{srcltx} 
\usepackage{marginnote,slashbox}
\usepackage{rotating}
\definecolor{darkblue}{rgb}{0.0,0.0,0.6}
\definecolor{darkgreen}{rgb}{0.0,0.6,0.0}
\usepackage[pdftex,colorlinks=true,urlcolor=darkblue,citecolor=darkblue,linkcolor=darkblue]{hyperref}
\usepackage[utf8]{inputenc}      
\usepackage{booktabs}            
\usepackage{url}                 
\usepackage[T1]{fontenc}         
\usepackage{algorithmic}         

\usepackage{calc}
\usepackage[notcite,notref]{showkeys}

\numberwithin{table}{section}    
\numberwithin{figure}{section}   
\numberwithin{equation}{section} 

\usepackage{my_latex_commands}

\newtheorem{assumption}[theorem]{Assumption}
\newtheorem{remark}[theorem]{Remark}

\renewcommand{\bl}[1]{\textcolor{black}{#1}}

\newcommand{\dl}{\delta \ell}
\newcommand{\e}{\varepsilon}

\renewcommand{\ll}{\bar \ell}

\newcommand{\lo}{\R^n}
\newcommand{\hon}{\R^n}
\newcommand{\hoon}{\R^n}

\newcommand{\dense}{\overset{d}{\embed}}

\newcommand{\hhyy}{H^1_0(0,T;\R^n)}

\newcommand{\llun}{L^2(0,T;\R^n)}

\newcommand{\hhyn}{H^1_0(0,T;\R^n)}
\newcommand{\F}{(\kappa \circ \HH)}
\renewcommand{\e}{\epsilon}
\begin{document}

\title{Optimality conditions for the control of a rate independent system with history variable}
\date{\today}
\author{Livia\ Betz\footnotemark[1]}
\renewcommand{\thefootnote}{\fnsymbol{footnote}}
\footnotetext[1]{Faculty of Mathematics, University of W\"urzburg,  Germany.}
\renewcommand{\thefootnote}{\arabic{footnote}}

\maketitle

\begin{abstract}
This paper addresses an optimal control problem governed by a rate independent  evolution involving an integral operator. Its particular feature is that the dissipation potential depends on the history of the state. Because of the non-smooth nature of the system, the application of standard adjoint calculus is excluded. We derive optimality conditions in qualified form by approximating the original problem by viscous  models. Though these problems preserve the non-smoothness, optimality conditions equivalent to the first-order necessary optimality conditions can be provided in the viscous case. Letting the viscous parameter vanish then yields an optimality system for the original control problem. If the optimal state at the end of the process  is not smaller than the desired state, the limit optimality conditions are complete.\end{abstract}

\begin{keywords}
history dependence, subdifferential inclusion, rate independence, optimal control, non-smooth optimization,    strong stationarity, viscous approximation, fatigue damage evolution \end{keywords}

\begin{AMS}
34G25,  49J40, 49K21, 74P99, 74R99
\end{AMS}
\section{Introduction}\label{sec:i} 
This paper is concerned with the derivation of optimality conditions for the control of the following history-dependent rate independent model:
\begin{equation}\label{eq:n}
-\partial_q \EE(t,q(t)) \in \partial_{\dot q} \RR (\HH(q)(t),\dot{q}(t)) \quad \ae (0,T),\quad q(0)=0.
  \end{equation}In \eqref{eq:n}, the symbol $\partial_{\dot q}$ denotes the convex subdifferential  with respect to the second variable.
  Thus, the above \textit{non-smooth differential inclusion}  is to be understood as:
  $$\dual{-\partial_q \EE(t,q(t))}{\eta-\dot{q}(t)}_{\lo} \leq \RR (\HH(q)(t),\eta)- \RR (\HH(q)(t),\dot{q}(t)) \quad \forall\,\eta \in \lo, \quad \ae (0,T).$$

   The stored energy $\EE:[0,T] \times \R^n \rightarrow \mathbb{R}$ is given by
\begin{equation}\label{def:e}\begin{aligned}
\EE(t, q)&:=\frac{\alpha}{2}\| q\|_{\R^n}^2-\dual{\ell(t)}{q}_{\R^n},
\end{aligned}\end{equation}where $\alpha>0$ is a fixed parameter.   
The time-dependent load $\ell:[0,T] \to \R^n$ appearing in \eqref{def:e}  will act later on as a control (see \eqref{eq:min} below). This  induces a certain state, that is  expressed in terms of the variable $q : [0,T] \to \R^n$. The  dissipation  functional  $\RR:\R^n \times  \R^n \rightarrow [0,\infty]$ is defined as \begin{equation}\label{def:r}
\RR(\zeta,\eta):=\left\{ \begin{aligned} \dual{\kappa(\zeta) }{  \eta}_{\R^n}   &\quad \text{if }\eta \in \CC,
\\\infty  &\quad \text{otherwise,}\end{aligned} \right.
\end{equation}where \begin{equation}\label{def:c}
\CC:=\{\eta \in \hon:\eta_i \geq 0, \ i=1,...,n \}
\end{equation}and $\kappa$ is a differentiable non-linearity, cf.\,Assumption \ref{assu}.\ref{it:2} for more details.
The  positive homogeneity w.r.t.\,the second argument of $\RR,$ (i.e., $\RR(\zeta,\gamma\eta)=\gamma \RR(\zeta,\eta)$ for all $\gamma \geq 0$ and for all $(\zeta,\eta)\in \R^n \times  \R^n$) shows that our model is rate independent. This means that  the solutions to \eqref{eq:n} are not affected by time-rescaling.


The essential aspect concerning the evolution in \eqref{eq:n} is the presence of the \textit{history} of the state. 
Indeed, the phenomenon of history dependence in connection with evolution inclusions has gained much interest in the past decades, see  e.g.\,\cite{mon} and the references therein. We also refer to the works \cite{oc_h_d, oc_sof} where existence of optimal control for such problems has been examined.
However, when it comes to the mathematical investigation of \textit{rate independent}  evolutions with a \textit{history} component (in terms of an integral operator), the literature is scant. The only works known to the author addressing this topic with regard to a rigorous analysis that involves existence and uniqueness of solutions are \cite{alessi} and \cite{ris}.

The structure of the equation \eqref{eq:n} is inspired by damage models with fatigue (cf.\, \cite{alessi, alessi1} and the references therein).   Therein, $\HH$ is an integral operator, also known as \textit{history operator} (Assumption \ref{assu}.\ref{it:1}). In the context of a damage evolution, $\HH$ models how the  damage experienced by the material affects its fatigue level. The degradation mapping $\kappa:\R \to [0,\infty)$ appearing in \eqref{def:r} indicates in which measure the fatigue affects the toughness  of the material. The latter is usually described by a fixed  (nonnegative) constant \cite{fn96, FK06}, while in the present model it changes  in time, depending on $\HH(q)$. To be more precise, the value of the  toughness of the body at time point $t \in [0,T]$ is given by  $\kappa(\HH(q))(t)$, cf.\ \eqref{eq:n} and \eqref{def:r}. Hence, damage  models with a history variable take into account the following crucial aspect: the occurrence of damage is favoured in regions where fatigue accumulates.

To the subdifferential inclusion in \eqref{eq:n} we associate  the following 
 optimal control problem:
\begin{equation}\label{eq:min}\tag{P}
 \left.
 \begin{aligned}
  \min \quad & j(q)+\frac{1}{2}\|q(T)-q_d\|^2_{\R^n}+\frac{1}{2}\|\ell\|^2_{H^1(0,T;\lo)}\\
  \text{s.t.} \quad & (\ell,q) \in {H_0^1(0,T;\lo)}\times {H_0^1(0,T;\hon)},
  \\\quad & q \text{ solves }
\eqref{eq:n} \text{ w.r.h.s.\,}\ell,
 \end{aligned}
 \quad \right\}
\end{equation}where $j: L^2(0,T;\R^n) \to \R$ is a smooth functional (Assumption \ref{assu:j}) and $q_d\in \R^n,\ q_d^i \geq 0, i=1,...,n, $ is the desired value of the state at the end of the process.

The main novelty in this paper arises from the fact that we aim at deriving optimality conditions for the control of a rate independent problem with \textit{state-dependent dissipation potential}. The particularity of such state equations  is that the dissipation functional  does  not  depend only on the rate, but also on the state itself \cite{mie_ros, BS} or it may depend on the time variable  in terms of a moving set \cite{kurzweil}. 
While  numerous rate independent models have been 
  addressed in the literature (see e.g.\,the references in \cite{mielke}), 
   only a few deal with a dissipation potential that has \textit{two} arguments \cite{mie_ros, BS, kurzweil, ris}. The additional dependence on the state gives in particular  rise to   difficulties when it comes to  the uniqueness of solutions. 
   
    The existence of optima for control problems governed by rate independent systems with \textit{state independent dissipation potential} has been investigated by various authors. We refer here only to \cite{k_st, kms} which focus on the (more difficult) infinite dimensional case with non-convex energies; see also the references therein. However, when it comes to  \textit{rate independent} problems with \textit{state dependent dissipation potential}, the only paper addressing the existence of optimal minimizers is \cite{ris}. Therein, the infinite dimensional counterpart of \eqref{eq:n} has been  analyzeded regarding unique solvability. Thus, the manuscript \cite{ris} provides the basis for our investigations here, as it ensures that the control-to-state map is single valued. It is beyond the scope of the present work to  make assertions about its directional differentiability and we refer  to \cite{brok_ch} where this matter (along with strong stationarity) is examined in the one-dimensional case in the context of rate independent systems with \textit{state independent dissipation potential}.
   
With the unique solvability of \eqref{eq:n} at hand, the challenge concerning  \eqref{eq:min} arises from the  \textit{non-smooth}  character of the state equation, which is  due to the non-differentiability of the dissipation functional $\RR$. This excludes the application of standard adjoint techniques for the derivation of first-order necessary conditions in form of optimality  systems.

The purpose of this paper is to  derive optimality conditions for  \eqref{eq:min} by means of an approximation that involves  control problems governed by \textit{viscous} evolutions, see \eqref{eq:q1} below. We emphasize that this approach is novel. The classical literature relies on the prominent smoothening technique for the derivation of  optimality systems in qualified form  proposed by \cite{barbu84}, where one replaces the non-smoothness by a suitable smooth function and then lets the smoothening parameter vanish in the respective KKT-system.
We refer here to \cite{sww, bk, w3 }, which established optimality conditions for the control of  rate independent systems \textit{with state independent dissipation potential} by this method. See also the more recent contribution \cite{col} where a time discretization scheme is employed.

The philosophy that solutions to rate independent systems  arise   as limits of sequences of solutions to viscous systems for some parameter $\epsilon$ approaching $0$ is well-known from settings that deal with non-convex energies \cite{mr15}. Though our energy is convex, we will also make 
use of  this idea to  investigate the optimal control of \eqref{eq:min} by starting with the (optimal control of the)  corresponding viscous system. This provides the advantage that it can be rewritten as a \textit{non-smooth ODE} for fixed $\epsilon$, see \eqref{eq:syst_diff1}. 
However, deriving necessary optimality conditions for such non-smooth problems is a challenging issue even in finite dimensions. In \cite{ScheelScholtes2000} a detailed overview of various optimality conditions of different strength was introduced. The most rigorous stationarity concept is \textit{strong stationarity}. Roughly speaking, the strong stationarity conditions involve an optimality system, that is equivalent to the purely primal conditions saying that the directional derivative of the reduced objective in feasible directions is nonnegative (which is referred to as B stationarity). Thus, strong stationarity  can be seen as the "non-smooth" counterpart of  the prominent KKT conditions.
By making use of the limited differentiability properties of the control-to-state map, it has been well established lately that, for certain classes of  non-smooth equations and viscous systems, strong stationary optimality conditions can be provided. Cf.\,\cite{st_coup, aos} (viscous systems) and  \cite{paper, frac} (time-dependent PDEs/ODEs). 

Once the strong stationary optimality conditions for the control of the history-dependent viscous model \eqref{eq:q1} are established, passing to the limit $\e \searrow 0$ will yield an optimality system, that could be classified as C stationary \cite{ScheelScholtes2000}. If the value of the optimal state at the end of the process is larger than the desired state, this optimality system is of strong stationary type (Remark \ref{rem:ft}), i.e., it is complete. During the vanishing viscosity process we will   be confronted with the \textit{history-dependence} of the state equation, which as we will see, gives rise to additional challenges in terms of showing uniform bounds and convergence analysis.
Finally, let us emphasize that the idea presented in this paper fully applies to the control of other rate independent processes with state dependent dissipation potential, where the state dependency may happen though a Nemytskii operator, see also Remark \ref{rem:ft}.

The paper is organized as follows. After introducing the notation, we recall in section \ref{2} some  findings that were recently established in \cite{ris} and \cite{aos}. These concern the unique solvability of \eqref{eq:n} and the properties of its viscous counterpart. Section \ref{2} ends with a new result regarding the convergence of the viscous approximation, see Proposition \ref{prop}. In section \ref{3} we then start our investigation of the optimal control of \eqref{eq:min} by showing in a classical manner  that its local optimal solutions  can be approximated by local minima of control problems governed by viscous models. For the latter we then establish strong stationary optimality conditions in Lemma \ref{thm:ss_qsep1}. The entirely novel part of this section starts with Proposition \ref{bound}, where uniform bounds with respect to the viscous parameter are provided for the involved multiplier and adjoint state. Here we make use of a previous crucial finding from \cite{ris}, namely   the uniform Lipschitz continuity of the viscous solution map (Lemma \ref{lem:lip_c}). Moreover, the term involving the history variable must be carefully investigated with respect to its regularity and uniform boundedness. With Proposition \ref{bound} at hand, we then derive a first optimality system in Theorem \ref{thm:os}, under a mild smoothness assumption on the mapping $\kappa$ (Assumption \ref{k''}). Again,  the presence of history brings out some challenges, this time in the convergence analysis. The section ends with an improvement of the optimality conditions from Theorem \ref{thm:os}, cf.\,Proposition \ref{prop:aoc}. This finding is established under some additional requirements. Finally, in section \ref{4} we make a comparison between our final optimality system and the expected one. To derive the latter, we resort to a formal Lagrange approach. Here we conclude that the derived optimality conditions are strong stationary (i.e., they do not lack any information) if the value of the optimal state is large enough at the end of the process.
The paper ends with some comments concerning other related models (Remark \ref{rem:ft}).

\subsection*{Notation}
Throughout the paper, $T > 0$ is a fixed final time and $n\in \N$ is the fixed dimension of the euclidean space. If $X$ and $Y$ are Banach spaces, the notation $X \embed \embed Y$ means that $X$ is \bl{compactly embedded} in $Y$, while $X \dense Y$ means that the embedding is dense.
We  use the abbreviation
\begin{equation*}
\begin{aligned}
W^{1,p}_0(0,T;\R^n)&:=\{y\in W^{1,p}(0,T;\R^n):y(0)=0\}, \quad p\in[1,\infty],
\end{aligned}
\end{equation*}
and the dual of $W_0^{1,1}(0,T;\hoon)$ is denoted by $W^{-1,\infty}(0,T;\hoon)$.
For the dual pairing between $X$ and its dual we write $\dual{\cdot}{\cdot}_X$, while $(\cdot,\cdot)_{Y}$ is the notation used for the scalar product in a Hilbert space $Y$. The euclidean product is however denoted by $\dual{\cdot}{\cdot}_{\R^n}.$ For the adjoint operator of a linear and bounded mapping $A$ we write  $A^\star$. Weak derivatives  are sometimes denoted by a dot. 

The symbol $\partial f$ stands for the convex subdifferential, see e.g.\ \cite{rockaf}.
The mapping $\II_\CC:\R^n \to \{0,\infty\}$ is the indicator functional of the set $\CC \subset \R^n$, i.e., $\II_\CC(y)=0,$ if $y \in \CC$ and $\II_\CC(y)=\infty,$ otherwise.

For a vector $\eta \in \R^n$,we write $\eta \geq 0$ if each component of $\eta$ is nonnegative.
Non-linearities $f:\R \to \R$  sometimes act on vectors, in which case they become vector-valued, by associating to each (real-valued) vector component of $\eta \in \R^n$ the  value $f(\eta_i), i=1,...,n$. For convenience, we denote them by the same symbol and from the context it will be clear which one is meant. The associated Nemytskii operators are also denoted by the same symbol.

Throughout the paper, $c,C>0$ are generic constants that depend only on the fixed physical parameters. 
To emphasize the dependence of a constant  on a certain fixed parameter $M$ we sometimes write $c(M).$

\section{The control-to-state map and its viscous approximation}\label{2}

We begin this section by presenting some  results that were established in  \cite{ris}
 regarding \eqref{eq:n} and that will be useful throughout the paper.

\begin{assumption}\label{assu}
For the mappings associated with history  in \eqref{eq:n} we require the following:
 \begin{enumerate}
  \item \label{it:1}  The \textit{history operator} $\HH:L^1(0,T;\lo) \to W^{1,1}(0,T;\lo)$ is given by
\begin{equation*}
 [0,T] \ni t \mapsto \HH(y)(t):=\int_0^t y(s) \,ds +y_0 \in \lo,\end{equation*}
where  $y_0\in \lo.$
\item \label{it:2}
The non-linear function  $\kappa: \R \to [0,\infty)$ is assumed to be Lipschitz continuous with Lipschitz constant $L_\kappa>0$ and  differentiable.  Moreover, $\kappa' \in W^{1,2}(\R).$
 \end{enumerate}
\end{assumption} 
Assumption \ref{assu} is tacitly supposed to be true throughout the entire paper, without mentioning it every time.

\begin{definition}[Control-to-state map]
In all what follows,
$$S:H^1_0(0,T;\hoon) \to \hhyy$$ denotes the solution operator of \eqref{eq:n}. Note that this is well-defined \cite[Thm.\,6.4]{ris}.
\end{definition}

\begin{lemma}[Equivalent formulation, {\cite[Prop.\,4.2]{ris}}]\label{prop:ris}
The  problem  \eqref{eq:n} is equivalent to
  \begin{equation}\label{eq:ic_ris}  
-\alpha q(t) +\ell(t)-\F(q)(t) \in \partial \II_\CC(\dot q(t)) \quad  \ae (0,T).\end{equation}
In particular, $-\alpha q(t) +\ell(t)-\F(q)(t)\leq 0$ for all $t\in [0,T]$.
\end{lemma}
 \begin{lemma}[Weak continuity of $S$, {\cite[Cor.\,7.1]{ris}}]\label{wc}
 The operator $S$ is weakly continuous from $H^{1}_0(0,T;\lo)$  to $H_0^{1}(0,T;\hon)$.
 \end{lemma} 

\subsection*{The history-dependent viscous evolution}\label{sec:0}
In the rest of the section we are concerned with the viscous version of  \eqref{eq:n} and its approximation properties.
This reads as follows:
\begin{equation}\label{eq:q1}
-\partial_q \EE(t,q(t)) \in \partial_{\dot q} \RR_\e (\HH(q)(t),\dot{q}(t))\quad \text{a.e.\,in }(0,T), \quad q(0) = 0,
  \end{equation}where the stored energy is given by \eqref{def:e}.
  The viscous dissipation functional $\RR_\e:\R^n \times  \R^n \rightarrow [0,\infty]$ is defined as \begin{equation}\label{def:r1}
\RR_\e(\zeta,\eta):=\left\{ \begin{aligned} \kappa(\zeta) \,  \eta \;dx +\frac{\e}{2}\|\eta\|^2_{\R^n}, &\quad \text{if }\eta \in \CC,
\\\infty,  &\quad \text{otherwise,}\end{aligned} \right.
\end{equation}
where  $\e>0$ is the viscosity parameter.

\begin{definition}[Solution operator of the viscous problem]
Throughout the paper,
$$S_\e:\llun \to \hhyy$$ denotes the solution map associated to \eqref{eq:q1}. Note that this is well-defined, Lipschitz continuous  and directionally differentiable {\cite[Sec 4.1]{aos}}.
\end{definition}


\begin{lemma}[Equivalence to an ODE, {\cite[Prop.\,1]{aos}}]\label{lem:ode1}
 The viscous problem \eqref{eq:q1} is equivalent to   \begin{equation}\label{eq:syst_diff1}
 \dot q(t)   = \frac{1}{\epsilon} \max\{-\alpha  q(t) +\ell(t)-\F(q)(t),0\} \quad \text{in }(0,T),  \quad q(0) = 0.
 \end{equation}
   \end{lemma} 

 \begin{lemma}[Uniform Lipschitz continuity of the viscous solution map, {\cite[Prop.\,7.2]{ris}}]\label{lem:lip_c}
 Let  $\e>0$ be fixed and let $\ell_1,\ell_2 \in H^1_0(0,T;\hoon)$ be given such that   $\|\ell_i\|_{H^1(0,T;\hoon)} \leq M, i=1,2,$ for some fixed $ M>0$. Then, it holds 
$$\|S_\e(\ell_1)-S_\e(\ell_2)\|_{C([0,T];\hon)} \leq c( M)\,\|\ell_1-\ell_2\|_{W^{1,1}(0,T;\hoon)}, $$where $c( M)>0$ is independent of $\e$.\end{lemma}

 \begin{proposition}[Convergence of the viscous approximation]\label{prop}
Let $\{\ell_\e\}\subset H^1_0(0,T;\lo)$ be a given subsequence. If $\ell_\e \weakly \ell $ in $H^1(0,T;\lo)$, then
 it holds 
$$S_\e(\ell_\e) \weakly S(\ell) \quad \text{ in } \hhyy, \  \text{as }\e \searrow 0.$$
Moreover, if $\ell_\e \to \ell $ in $H^1(0,T;\lo)$, one has 
$$S_\e(\ell_\e) \to S(\ell) \quad \text{ in } \hhyy, \  \text{as }\e \searrow 0.$$
\end{proposition} 

 \begin{proof}Throughout the proof, we abbreviate for simplicity $q_\e:=S_\e(\ell_\e)$.
According to 
 \cite[Prop.\,3.6]{ris}, it holds \begin{equation}\label{eq:est1}
\|q_\e\|_{\hhyy} \leq C\,\|\ell_\e\|_{H^1(0,T;\hoon)},
 \end{equation}where $C>0$ is independent of $\e,$ so that we can extract a weakly convergent subsequence with weak limit $\hat q \in \hhyy.$ 
 By \cite[Prop.\,2.7]{ris}, $q_\e$ satisfies
 \begin{equation}\label{for_l}\begin{aligned}
\int_0^t  \kappa(\HH(q_\e)(\tau))\dot q_\e (\tau)\,  d \tau& +\e \int_0^t \| \dot q_\e (\tau)\|^2_{\R^n} \,d \tau 
+\frac{\alpha}{2}\| q_\e(t)\|_{\R^n}^2-{\ell_\e(t)}{q_\e(t)} 
\\&\qquad \qquad \qquad \qquad =- \int_0^t \hspace{-0.15cm} {\dot \ell_\e(\tau)}{q_\e(\tau)} d \tau  \quad \forall\,0\leq t \leq T.\end{aligned}\end{equation}
 Passing to the limit in \eqref{for_l}, where one uses the compact embedding $H^1(0,T;\hoon) \embed \embed C([0,T];\R^n)$ and Assumption \ref{assu}, yields that $\hat q$ satisfies 
  \begin{equation*}\begin{aligned}
\int_0^t  \kappa(\HH(\hat q)(\tau))\dot {\hat q} (\tau)\,  d \tau& +\frac{\alpha}{2}\| \hat q(t)\|_{\R^n}^2-{\ell(t)}{\hat q(t)} 
 =- \int_0^t \hspace{-0.15cm} {\dot \ell(\tau)}{\hat q(\tau)} d \tau  \quad \forall\,0\leq t \leq T.\end{aligned}\end{equation*}
By arguing as in the proof of \cite[Prop.\,4.3]{ris}, we finally obtain that $\hat q$ solves  
 \eqref{eq:n} with right hand side $\ell$, i.e., $\hat q=S(\ell)$.
To show the second assertion, we recall the identity \begin{equation}\label{eq:1}
  \alpha \| \dot q_\e(t)\|_{\lo}^2 +\frac{\e}{2} \frac{d}{dt}  \| \dot q_\e(t)\|_{\lo}^2=\dual{\dot \ell_\e (t)-\F'(q_\e)(\dot q_\e)(t)}{\dot q_\e(t)}_{\hon} \quad  \ae (0,T),\end{equation}which was established in the proof of \cite[Prop.\,3.6]{ris}.
 Integrating over time implies
   \begin{equation}\label{eq:syst_dif22}
\begin{aligned}
\alpha \| q_\e\|_{H^1_0(0,T;\hon)}^2 + \int_0^T\dual{\F'(q_\e)(\dot q_\e)(t)}{\dot q_\e(t)}_{\hon} \,dt  \leq \int_0^T\dual{\dot \ell_\e (t)}{\dot q_\e(t)}_{\hon} \,dt.
  \end{aligned}\end{equation}
In light of Assumption \ref{assu}, we have  \begin{equation}\label{kap}
\begin{aligned}
\int_0^T \dual{\F'(y)(\dot y)(t)}{\dot y(t)}_{\lo}\,dt 
=\int_0^T \dual{\kappa'(\HH(y)(t))(y(t))}{ \dot y(t)}_{\lo}\,dt
\end{aligned}
  \end{equation}for all $y \in H_0^1(0,T;\R^n)$, cf.\,the proof of \cite[Lem.\,3.2]{ris}. Since 
\begin{equation}\label{strong}
q_\e \weakly q \quad \text{ in } \hhyy \  \text{as }\e \searrow 0,\end{equation}
where $q:=S(\ell),$ we deduce
 $$\HH(q_\e) \to \HH(q) \quad \text{in }C([0,T];\R^n),$$in view of the definition of $\HH$. Further, we have $\kappa'\in C^{0,1/2}[-M,M],$ for each $M >0,$ as a result of Assumption \ref{assu}. Thus, setting $M:=\sup_{\e>0} \|\HH(q_\e)\|_{C([0,T];\R^n)}+1,$ implies 
  \begin{equation}\label{k'_cont}
  \kappa'(\HH(q_\e)) \to \kappa'(\HH(q)) \quad \text{in }C([0,T];\R^n).
  \end{equation}
 By relying  on \eqref{kap}, \eqref{strong} and $H_0^1(0,T;\R^n) \embed \embed C([0,T];\R^n)$ we then arrive at the convergence 
 
  $$
  \int_0^T\dual{\F'(q_\e)(\dot q_\e)(t)}{\dot q_\e(t)}_{\hon} \,dt \to \int_0^T\dual{\F'(q)(\dot q)(t)}{\dot q(t)}_{\hon} \,dt.$$
 Now we go back to  \eqref{eq:syst_dif22}, where we use again \eqref{strong} as well as  the assumption $\ell_\e \to \ell $ in $H^1(0,T;\lo)$. This gives in turn     \begin{equation}\label{use}
\begin{aligned}
\alpha \| q\|_{H^1_0(0,T;\hon)}^2
&\leq \liminf \alpha \| q_\e\|_{H^1_0(0,T;\hon)}^2
\\&\leq \limsup \alpha \| q_\e\|_{H^1_0(0,T;\hon)}^2  
\\& \leq \int_0^T\dual{\dot \ell (t)}{\dot q(t)}_{\hon} \,dt-\int_0^T\dual{\F'(q)(\dot q)(t)}{\dot q(t)}_{\hon} \,dt 
\\&\quad =\alpha \| q\|_{H^1_0(0,T;\hon)}^2
  \end{aligned}\end{equation}We underline that the last inequality follows by the exact same arguments as \eqref{eq:1}, see the proof of \cite[Prop.\,3.6]{ris}. In view of \eqref{strong} and \eqref{use}, we get 
$$
q_\e \to q \quad \text{ in } \hhyy \  \text{ as }\e \searrow 0,$$
which is the desired assertion.
 \end{proof}
  
\section{Optimal control}\label{3}
This section focuses on the investigation of the optimal control of the rate independent  model with history \eqref{eq:n}. For convenience, we recall the associated optimization problem:
 \begin{equation}\label{eq:oc}\tag{P}
 \left.
 \begin{aligned}
  \min \quad & j(q)+\frac{1}{2}\|q(T)-q_d\|^2_{\R^n}+\frac{1}{2}\|\ell\|^2_{H^1(0,T;\lo)}\\
  \text{s.t.} \quad & (\ell,q) \in {H_0^1(0,T;\lo)}\times {H_0^1(0,T;\hon)},
  \\\quad & q \text{ solves }
\eqref{eq:n} \text{ w.r.h.s.\,}\ell,
 \end{aligned}
 \quad \right\}
\end{equation}where $q_d\in \R^n $ is the desired value of the state at the end of the process. 
In the sequel, the first part of the objective that acts on the state is supposed to fulfil the following
\begin{assumption}\label{assu:j}
  The functional $j: L^2(0,T;\R^n) \to \R$ is 
  continuously differentiable.  \end{assumption}
  Assumption \ref{assu:j} is tacitly assumed in all what follows, without mentioning it every time.
\begin{proposition}
The optimization problem \eqref{eq:oc} admits at least one solution.
\end{proposition}
\begin{proof}
The statement follows by the direct method of calculus of variations in combination with Lemma \ref{wc}.
\end{proof}

 In order to derive necessary optimality conditions in qualified form,
 we start by
proving that each local optimum $\ll$ of \eqref{eq:oc} can be approximated via local optima of the control problem:
 \begin{equation}\label{eq:min_e}\tag{$P_\e$}
 \left.
 \begin{aligned}
  \min \quad & j(q)+\frac{1}{2}\|q(T)-q_d\|^2_{\R^n}+\frac{1}{2}\|\ell\|^2_{H^1(0,T;\lo)}+\frac{1}{2}\|\ell-\bar \ell\|^2_{H^1(0,T;\lo)}\\
  \text{s.t.} \quad & (\ell,q) \in {H_0^1(0,T;\lo)}\times {H_0^1(0,T;\hon)},
  \\\quad & q \text{ solves }
\eqref{eq:q1} \text{ w.r.h.s.\,}\ell.
 \end{aligned}
 \quad \right\}
\end{equation}

We point out that, unlike in the classical literature, the approximating control problems are non-smooth. However, one can establish strong stationary optimality conditions for the control of \eqref{eq:min_e}, meaning that the approximating optimality systems are as strong as the KKT conditions in the smooth case (Lemma \ref{thm:ss_qsep1} and Remark \ref{rem:ss} below). In fact, keeping the  non-smoothness in the approximating setting enables us later on to prove a C stationary like optimality system. If the optimal state at the end of the process is large enough, we even arrive at the complete optimality conditions,  see Remark \ref{rem:st} below.

\begin{lemma}\label{approx}
For each local optimum $\ll$ of \eqref{eq:oc} there exists a sequence $\{\ell_\e\}_{\e>0}$ of local minimizers of $\{ \eqref{eq:min_e}\}_{\e>0}$ so that
 \begin{equation}\label{l_conv1}
\ell_{\e} \to \ll \quad  \text{in }H_0^1(0,T;\lo), \end{equation}
and
\begin{equation}\label{y_conv1}
  S_\e(\ell_{\e}) \to S(\ll)  \quad \text{in } \hhyy.
  \end{equation}
\end{lemma}
\begin{proof}The result is shown by classical arguments, see for instance \cite{barbu84, tiba}. For convenience and completeness, we give a detailed proof here.
 Let $B_{H^1_0(0,T;\lo )}(\ll,  \rho),\rho>0,$ be the ball of local optimality of $\ll$. By the direct method of calculus of variations, 
we see that \begin{equation}\label{min_e}
 \left.
 \begin{aligned}
  \min_{\ell \in B_{H^1_0(0,T;\lo )}(\ll,  \rho)} \quad & j(q)+\frac{1}{2}\|q(T)-q_d\|^2_{\R^n}+\frac{1}{2}\|\ell\|^2_{H^1(0,T;\lo)}+\frac{1}{2}\|\ell-\bar \ell\|^2_{H^1(0,T;\lo)}\\
  \text{s.t.} \quad & q \text{ solves }
\eqref{eq:q1} \text{ w.r.h.s.\,}\ell
 \end{aligned}
 \quad \right\}
\end{equation}
admits a global solution $\ell_{\e} \in H^1_0(0,T;\lo)$; note that here we use the compact embedding $H_0^1(0,T;\R^n) \embed \embed L^2(0,T;\R^n) $, the Lipschitz continuity of $S_\e:L^2(0,T;\R^n) \to H_0^1(0,T;\R^n)$ and the continuity of $j$.  Since $\ell_{\e}  \in B_{H^1_0(0,T;\lo)}(\ll,  \rho),$ we can select a subsequence with 
\begin{equation}\label{l}
\ell_{\e}  \weakly \widetilde \ell \quad \text{in }H^1_0(0,T;\lo),
\end{equation}
where $\widetilde \ell  \in B_{H^1_0(0,T;\lo)}(\ll,  \rho).$
For simplicity, we abbreviate in the following   \begin{subequations} \begin{gather}
\JJ(\ell):=j(S(\ell))+\frac{1}{2}\|S(\ell)(T)-q_d\|^2_{\R^n}+\frac{1}{2}\|\ell\|_{H^1(0,T;\lo)}^2 ,\label{jj}\\
\JJ_{\e}(\ell):=j(S_\e(\ell))+\frac{1}{2}\|S_\e(\ell)(T)-q_d\|^2_{\R^n}+\frac{1}{2}\|\ell\|_{H^1(0,T;\lo)}^2+\frac{1}{2}\|\ell-\ll\|_{H^1(0,T;\lo)}^2  \label{jn}
\end{gather}\end{subequations} \normalsize for all $\ell \in H^1(0,T;\lo)$.
Due to Proposition \ref{prop},  it holds 
  \begin{equation}\label{j}
  \JJ(\ll)=\lim_{\e \to 0} j(S_\e(\ll))+\frac{1}{2}\|S_\e(\ll)(T)-q_d\|^2_{\R^n}+\frac{1}{2}\|\ll\|_{H^1(0,T;\lo)}^2\overset{\eqref{jn}}{=}\lim_{\e \to 0}  \JJ_{\e}(\ll) \geq \limsup_{\e \to 0}   \JJ_{\e}(\ell_{\e}),
  \end{equation}\normalsize where for the last inequality we relied on the fact that $\ell_{\e}$ is a global minimizer of \eqref{min_e} and that $\ll$ is admissible for \eqref{min_e}. In view of  \eqref{jn}, \eqref{j} can be continued as 
   \begin{equation}\label{j1}\begin{aligned}
  \JJ(\ll) &\geq \limsup_{\e \to 0}   j(S_\e(\ell_\e))+\frac{1}{2}\|S_\e(\ell_\e)(T)-q_d\|^2_{\R^n}+\frac{1}{2}\|\ell_\e\|_{H^1(0,T;\lo)}^2+\frac{1}{2}\|\ell_{\e}-\ll\|_{H^1(0,T;\lo)}^2
   \\&\quad   \geq \liminf_{\e \to 0}   j(S_\e(\ell_\e))+\frac{1}{2}\|S_\e(\ell_\e)(T)-q_d\|^2_{\R^n}+\frac{1}{2}\|\ell_\e\|_{H^1(0,T;\lo)}^2+\frac{1}{2}\|\ell_{\e}-\ll\|_{H^1(0,T;\lo)}^2
      \\&\qquad   \geq    j(S(\widetilde \ell))+\frac{1}{2}\|S(\widetilde \ell)(T)-q_d\|^2_{\R^n}+\frac{1}{2}\|\widetilde \ell\|_{H^1(0,T;\lo)}^2    +\frac{1}{2}\|\widetilde \ell-\ll\|_{H^1(0,T;\lo)}^2
  \geq   \JJ(\ll),
\end{aligned}  \end{equation}where we used Proposition \ref{prop} in combination with \eqref{l} and Assumption \ref{assu:j}; note that for the last inequality in \eqref{j1} we employed the fact that $\widetilde \ell \in B_{H^1_0(0,T;\lo )}(\ll,  \rho)$.
From \eqref{j1} we obtain that
\begin{align*}
\JJ(\ll)& = \lim_{\e \to 0}     j(S_\e(\ell_\e))+\frac{1}{2}\|S_\e(\ell_\e)(T)-q_d\|^2_{\R^n}+\frac{1}{2}\|\ell_\e\|_{H^1(0,T;\lo)}^2+\frac{1}{2}\|\ell_{\e}-\ll\|_{H^1(0,T;\lo)}^2
\\&= \lim_{\e \to 0}     j(S_\e(\ell_\e))+\frac{1}{2}\|S_\e(\ell_\e)(T)-q_d\|^2_{\R^n}+\frac{1}{2}\|\ell_\e\|_{H^1(0,T;\lo)}^2,
\end{align*}
whence   the convergence  
 \begin{equation*}
\ell_{\e} \to \ll \quad  \text{in }H^1(0,T;\lo) \end{equation*}
follows.
As a consequence, Proposition \ref{prop} gives in turn 
\begin{equation*}
  S_\e(\ell_{\e}) \to S(\ll)  \quad \text{in } \hhyy.
  \end{equation*}
A classical argument \cite{barbu84} finally shows that $\ell_\e$ is a local minimizer of \eqref{eq:min_e} for $\e>0$ sufficiently small.
\end{proof}

Strong stationary optimality conditions for the control of the infinite-dimensional  viscous  problem with history  \eqref{eq:q1} have been established in  {\cite[Thm.\,12]{aos}}. However, the control space in that case was $L^2(0,T)$ instead of $H_0^1(0,T)$.
For convenience of the reader and for the sake of completeness, we give a detailed  proof below.
\begin{lemma}[Strong stationarity for the optimal control of the viscous  model ]\label{thm:ss_qsep1}
 Let $ \ell_\e \in \hhyn$ be locally optimal for \eqref{eq:min_e} with associated state  $
q_\e \in H^{2}_0(0,T;\hoon).$   Then, there exists a unique adjoint state
$
  \xi_\e \in H^{1}(0,T;\hoon)$ and a unique multiplier $\lambda_\e \in L^{\infty}(0,T;\R^n)$ such that the following system is satisfied 
 \begin{subequations}\label{eq:strongstat_q1}
 \begin{gather}
  -\dot \xi_\e+\alpha \lambda_\e+[\F'( q_\e)]^\star \lambda_\e= j'( q_\e) \  \text{ in  }L^{2}(0,T;\hoon), \quad \xi_\e(T) =q_\e(T)-q_d,\label{eq:adjoint1_q1} \\[1mm]
 \left. \begin{aligned} 
  \lambda_\e^i(t)&=\frac{1}{\epsilon} \raisebox{3pt}{$\chi$}_{\{   z_\e^i>0\}}(t) \xi^i_\e(t) \quad \text{a.e.\ where  } z_\e^i(t) \neq 0,\\
 \lambda_\e^i(t) &\in \big[0,\frac{1}{\epsilon} \xi^i_\e(t) \big]  \quad \text{a.e.\ where  } z_\e^i(t) = 0,\quad \forall\,i=1,...,n,\end{aligned}\right\}\label{eq:signcond_q}
\\[1mm] \dual{\lambda_\e}{v}_{H^1(0,T;\lo)}+(\ell_\e,v)_{H^1(0,T;\lo)}+(\ell_\e-\ll,v)_{H^1(0,T;\lo)}=0\ \ \forall\, v \in H_0^1(0,T;\lo),\label{eq:grad_q1}
 \end{gather} 
 \end{subequations}    where we abbreviate $ z_\e:=-\alpha  q_\e+\ell_\e-\F( q_\e )$.\end{lemma}

\begin{proof}We first observe that the higher regularity of the state is due to \cite[Prop.\,3.6]{ris}.
According to \cite[Prop.\,3]{aos}, $S_\e$ is directionally differentiable. 
Its directional derivative $\delta q_\e:=S_\e'(\ell_\e;v)$ at  $\ell_\e $ in direction $v \in \llun$ is the unique solution of 
\begin{equation}\label{eq:syst_dif}\begin{aligned}
 \dot {\delta q_\e}(t) & = \frac{1}{\epsilon} \max\,'(z_\e(t);[-\alpha  \delta q_\e(t) +v(t)-\F'(q_\e)(\delta q_\e)(t)]) \quad \text{ a.e.\ in } (0,T),  \\ \delta q_\e(0) &= 0.
\end{aligned} \end{equation}
To see this, we refer to reader to \cite[Lem.\,7, Prop.\,2]{aos}; note that 
$$\omega-P_{\partial I_\CC(0)}\omega=\max(\omega,0) \quad \forall\,\omega \in \R^n$$where $P_{\partial I_\CC(0)}:\R^n \to \R^n$ stands for the projection operator on the set $\partial I_\CC(0):=\{\mu \in \R^n: \mu_i \leq 0, \ i=1,...,n\}$.

As $ \ell_\e \in \hhyn$ is locally optimal for \eqref{eq:min_e}, it satisfies the first order necessary optimality condition
\begin{equation}\label{eq:vi}
j'(q_\e) S_\e'(\ell_\e;v)  +\dual{q_\e(T)-q_d}{S_\e'(\ell_\e;v)(T)}_{\R^n}+(\ell_\e,v)_{H^1(0,T;\lo)}+(\ell_\e-\ll,v)_{H^1(0,T;\lo)}\geq 0\end{equation}
for all $v \in H_0^1(0,T;\lo)$. Since there is a constant $K>0$, independent of $\delta \ell$, so that 
$$
\|S_\e'( \ell_\e;v)\|_{L^2(0,T;\hon)} \leq K\, \|v\|_{L^2(0,T;\hoon)} \quad \forall\,\dl \in L^2(0,T;\hoon),$$see \cite[Lem.\,8]{aos}, Hahn-Banach theorem gives in turn that there exists $\lambda_\e \in L^{2}(0,T;\R^n)$ that satisfies \eqref{eq:grad_q1}. Further, 
we define 
$$ \xi_\e(t):=q_\e(T)-q_d+\int_t^T (-\alpha \lambda_\e-[\F'( q_\e)]^\star \lambda_\e+ j'( q_\e))(s)\,ds \quad \forall\,t\in [0,T],$$
so that \eqref{eq:adjoint1_q1} is true; note that $\xi_\e$ has the desired regularity.
Now, we proceed towards showing \eqref{eq:signcond_q}, by testing \eqref{eq:adjoint1_q1} with $\delta q_\e=S_\e'(\ell_\e;v)$ and \eqref{eq:syst_dif} with $\xi_\e$. In view of \eqref{eq:vi}, we arrrive at
\begin{equation}\begin{aligned}
&j'(q_\e) \delta q_\e  +\dual{q_\e(T)-q_d}{\delta q_\e(T)}_{\R^n}+(\ell_\e,v)_{H^1(0,T;\lo)}+(\ell_\e-\ll,v)_{H^1(0,T;\lo)}
\\&=( -\dot \xi_\e+\alpha \lambda_\e+[\F'( q_\e)]^\star \lambda_\e,\delta q_\e)_{L^2(0,T;\hoon)}+\dual{q_\e(T)-q_d}{\delta q_\e(T)}_{\R^n}- \dual{\lambda_\e}{v}_{H^1(0,T;\lo)}
\\&=(\dot \delta q_\e,\xi_\e)_{L^2(0,T;\hoon)}+(\alpha \lambda_\e+[\F'( q_\e)]^\star \lambda_\e,\delta q_\e)_{L^2(0,T;\hoon)}- \dual{\lambda_\e}{v}_{H^1(0,T;\lo)}
\\&=(\frac{1}{\epsilon} \max\,'(z_\e;[-\alpha  \delta q_\e +v-\F'(q_\e)(\delta q_\e)]),\xi_\e)_{L^2(0,T;\hoon)}
\\&\quad -(\lambda_\e, [-\alpha \delta q_\e+v-\F'( q_\e) (\delta q_\e)])_{L^2(0,T;\hoon)}
\geq 0\ \ \forall\, v \in H_0^1(0,T;\lo).
\end{aligned} \end{equation}
A density argument based on the embedding $H_0^1(0,T;\lo)\dense L^2(0,T;\lo)$ and the continuity of $\max\,'(z_\e;\cdot):L^2(0,T;\R^n) \to L^2(0,T;\R^n)$ then shows that
\begin{equation}\begin{aligned}\label{test}
&(\frac{1}{\epsilon} \max\,'(z_\e;[-\alpha  \delta q_\e +v-\F'(q_\e)(\delta q_\e)]),\xi_\e)_{L^2(0,T;\hoon)}
\\&\quad -(\lambda_\e, [-\alpha \delta q_\e+v-\F'( q_\e) (\delta q_\e)])_{L^2(0,T;\hoon)}
\geq 0\ \ \forall\, v \in L^2(0,T;\lo).
\end{aligned} \end{equation}
Now let $\eta \in L^2(0,T;\lo)$ be arbitrary but fixed. Then, if we define 
$$\rho(t):=\int_0^t \frac{1}{\epsilon} \max\,'(z_\e(s);\eta(s))\,ds \quad \forall\,t \in [0,T],$$
and $$\dl:=\eta+\alpha \rho +\F'(q_\e)(\rho) \in L^2(0,T;\lo),$$ we observe that $\rho=S_\e'(\ell_\e;\dl),$ in view of the unique solvability of \eqref{eq:syst_dif}. Thus, testing  \eqref{test} with $\dl$ gives in turn
\begin{equation}\begin{aligned}
(\frac{1}{\epsilon} \max\,'(z_\e;\eta),\xi_\e)_{L^2(0,T;\hoon)}
-(\lambda_\e, \eta)_{L^2(0,T;\hoon)}
\geq 0\ \ \forall\, \eta \in L^2(0,T;\lo).
\end{aligned} \end{equation}
By  testing with $\eta \geq 0$ and by 
 employing the fundamental lemma of calculus of variations combined with the positive 
homogeneity of the directional derivative with respect to the direction we deduce
\begin{equation*}
  \lambda_\e^i(t) \leq \frac{1}{\epsilon}\max\,'_+(   z_\e^i(t)) \xi_\e^i(t) \quad \text{a.e.\ in  } (0,T),\quad \forall\,i=1,...,n.
  \end{equation*}
In an analogous way, testing with $\eta \leq 0$ implies
\begin{equation*}
  \lambda_\e^i(t) \geq \frac{1}{\epsilon}\max\,'_-(   z_\e^i(t)) \xi_\e^i(t) \quad \text{a.e.\ in  } (0,T),\quad \forall\,i=1,...,n.
  \end{equation*}
  Note that $\max\,'_+$ and $\max\,'_-$ denote the right- and left-sided derivative of $\max$.
  Differentiating between the cases $\{z_\e^i>0\}$, $\{z_\e^i<0\}$ and $\{z_\e^i=0\}$ then leads to \eqref{eq:signcond_q}.
Now we may also conclude the $L^\infty(0,T;\R^n)-$regularity of $\lambda_\e$, since $\xi_\e \in H^1(0,T;\R^n) \embed L^\infty(0,T;\R^n)$. This completes the proof.
\end{proof}
\begin{remark}\label{rem:ss}
The optimality system in Lemma \ref{thm:ss_qsep1}  is indeed of strong stationary type, cf.\,\cite[Thm.\,13]{aos}. This means that, if for a given $\ell_\e \in H^1(0,T;\R^n)$ with associated state $q_\e$, there exists $(\xi_\e,\lambda_\e)$ so that \eqref{eq:strongstat_q1} is satisfied, then $\ell_\e$ fulfills the first order necessary optimality condition:
\begin{equation}\begin{aligned}
j'(q_\e) S_\e'(\ell_\e;v) & +\dual{q_\e(T)-q_d}{S_\e'(\ell_\e;v)(T)}_{\R^n}
\\&+(\ell_\e,v)_{H^1(0,T;\lo)}+(\ell_\e-\ll,v)_{H^1(0,T;\lo)}\geq 0\ \ \forall\, v \in H_0^1(0,T;\lo).
\end{aligned}
\end{equation}
In particular, if the set $\{z_\e^i=0\}$ has measure zero for each $i=1,...,n$, then \eqref{eq:strongstat_q1} reduces to the classical KKT conditions.
\end{remark}
\begin{proposition}[Uniform bounds]\label{bound}
Let $\ll$ be a local optimum of \eqref{eq:oc} and let $\{\ell_\e\}_{\e>0}$ be a sequence of local minimizers of $\{ \eqref{eq:min_e}\}_{\e>0}$ for which \eqref{l_conv1} is true. Then, 
there exists $C>0$, independent of $\e$, so that
\begin{equation}\label{le}
\|\lambda_\e\|_{W^{-1,\infty}(0,T;\hon)} \leq C,
\end{equation}
\begin{equation}\label{xe}
\| \xi_\e\|_{L^{\infty}(0,T;\hoon)} \leq C.
\end{equation}
where $\lambda_\e$ and $\xi_\e$ are given by Lemma \ref{thm:ss_qsep1}.
 \end{proposition} 
 \begin{proof}
 Let $v \in H^1_0(0,T;\R^n)$ be arbitrary, but fixed. In view of \eqref{l_conv1}, there exists $c>0$, independent of $\e$, so that $\|\ell_\e\|_{H^{1}_0(0,T;\hon)} \leq c$, 
and by applying Lemma \ref{lem:lip_c} with $M:=c+\|v\|_{H^1_0(0,T;\R^n)}$, we have  
$$\Big\|\frac{S_\e(\ell_\e+\tau v)-S_\e(\ell_\e)}{\tau}\Big\|_{C([0,T];\hon)} \leq c( M)\,\|v\|_{W^{1,1}(0,T;\hoon)} \quad \forall\,\tau \in (0,1),$$where $c(M)$ is independent of $\e$.  
As $S_\e:\llun \to \hhyy$ is directionally differentiable {\cite[Prop.\,3]{aos}}, we may pass to the limit $\tau \searrow 0$, from which we infer 
$$
\|S_\e'( \ell_\e; v)\|_{C([0,T];\hon)} \leq c(M)\, \|v\|_{W^{1,1}(0,T;\hoon)} \quad \forall\,v \in H^1_0(0,T;\R^n).$$ By employing the first order necessary optimality condition \eqref{eq:vi} and \eqref{eq:grad_q1}, we arrive at $$ \dual{-\lambda_e}{v}_{H^{1}(0,T;\hoon)}  \leq \|j'(S_\e( \ell_\e))\|_{L^1(0,T;\R^n)}\|S_\e'( \ell_\e; v)\|_{C([0,T];\hon)} \leq C\, \|v\|_{W^{1,1}(0,T;\hoon)}$$for all $v \in H^1_0(0,T;\R^n),$ where $C>0$ is independent of $\e$. Note that here we also used \eqref{y_conv1} and Assumption \ref{assu:j}. This implies the desired uniform bound for $\{\lambda_\e\},$ since $H^1_0(0,T;\R^n) \dense W_0^{1,1}(0,T;\hoon).$

To show the desired result for $\{\xi_\e\},$ we estimate the terms appearing in  \eqref{eq:adjoint1_q1}. 
We first observe that 
$$\frac{d}{dt} \kappa'(\HH(q_\e))(t)=\kappa''(\HH(q_\e)(t))(q_\e(t)) \quad \forall\, t \in [0,T],$$ in view of Assumption \ref{assu} and chain rule for composition of Sobolev functions, see e.g.\,\cite{mm_serrin}. 
This leads to 
\begin{equation}
\begin{aligned}
\|\kappa'(\HH( q_\e))\|_{W^{1,2}(0,T;\R^n)}&\leq \|\kappa'(\HH( q_\e))\|_{L^{2}(0,T;\R^n)}+\|\kappa''(\HH( q_\e))(q_\e)\|_{L^{2}(0,T;\R^n)}
\\&\leq L_\kappa \|q_\e\|_{L^{1}(0,T;\R^n)}+\|\kappa''\|_{L^{2}(\R)}\|q_\e\|_{L^{\infty}(0,T;\R^n)}
\\&\leq c,
\end{aligned}
\end{equation}where $c>0$ is independent of $\e$. Note that here we relied on \eqref{y_conv1} and  Assumption \ref{assu}.\ref{it:2}.
Let now $v \in W^{1,1}_0(0,T;\hon)$ be arbitrary but fixed. In light of the above, we infer
\begin{equation}\label{khb}
\begin{aligned}
\|\F'( q_\e)v\|_{W^{1,1}(0,T;\hoon)}& \leq \|\kappa'(\HH( q_\e))\|_{W^{1,2}(0,T;\R^n)}\|\HH'( q_\e)v\|_{W^{1,2}(0,T;\R^n)}
\\&\leq c\,\|v\|_{L^{2}(0,T;\hon)}
\\&\leq c\,\|v\|_{W^{1,1}(0,T;\hon)},
\end{aligned}
\end{equation}since $\HH'( q_\e)(v)=\int_0^\cdot v(s)\,ds$, cf.\,Assumption \ref{assu}.\ref{it:1}. We further notice that $$\F'( q_\e)v \in {W^{1,1}_0(0,T;\hoon)},$$ as $\HH'( q_\e)(v)(0)=0.$
Now let $\varphi \in L^1(0,T;\R^n)$ be arbitrary but fixed and define 
$$\widehat v:=\int_0^{\cdot} \varphi (s)\,ds.$$ Testing with $\widehat v \in {W^{1,1}_0(0,T;\hoon)}$ in \eqref{eq:adjoint1_q1}, where one uses the integration by parts formula, finally results in 
\begin{equation}\begin{aligned}
  \dual{ \xi_\e}{\varphi}_{L^1(0,T;\R^n)}&=
  -\dual{\dot \xi_\e}{\widehat v}_{W^{1,1}_0(0,T;\hon)}+\dual{q_\e(T)-q_d}{\widehat v(T)}_{\R^n}
 \\& \quad \leq \alpha \|\lambda_\e\|_{W^{-1,\infty}(0,T;\hon)}\|\widehat v\|_{W^{1,1}(0,T;\hon)}
 \\& \qquad +\|\lambda_\e\|_{W^{-1,\infty}(0,T;\hon)}\|\F'( q_\e)\widehat v\|_{W^{1,1}(0,T;\hoon)}
 \\& \qquad +\|j'( q_\e)\|_{L^{2}(0,T;\hon)}\|\widehat v\|_{W^{1,1}(0,T;\hon)} 
 \\& \qquad +\|q_\e-q_d\|_{C([0,T];\R^n)}\|\varphi\|_{L^1(0,T;\R^n)}
 \\& \quad  \leq C\,(\|\widehat v\|_{W^{1,1}(0,T;\hon)}+\|\varphi\|_{L^1(0,T;\R^n)}),
\end{aligned} \end{equation}where $C>0$ is independent of $\e$.
This is due to \eqref{le}, \eqref{khb}, \eqref{y_conv1} and  Assumption \ref{assu}.\ref{it:2}. Since $\varphi \in L^1(0,T;\R^n)$ was arbitrary and $\|\widehat v\|_{W^{1,1}(0,T;\hon)} \leq c \|\varphi\|_{L^{1}(0,T;\hon)} $, the proof is now complete.
    \end{proof}

    
In order to be able to drive the viscosity parameter $\e$ to 0 in \eqref{eq:strongstat_q1}, we need a smoothness assumption.
\begin{assumption}\label{k''}
The function $\kappa:\R \to [0,\infty)$ is twice continuously differentiable.
\end{assumption}
\begin{theorem}[Optimality conditions]\label{thm:os}
Suppose that Assumption \ref{k''} is true.
 Let $\bar \ell \in \hhyn$ be locally optimal for \eqref{eq:oc} with associated state  $
  \bar q \in \hhyy.$   Then, there exists an adjoint state
$
  \xi \in L^\infty(0,T;\hoon)$ and a unique multiplier $\lambda \in W^{-1,\infty}(0,T;\R^n)$ such that the following system is satisfied 
 \begin{subequations}\label{eq:opt_syst}
 \begin{gather} 
 \int_0^T  \xi(t) \dot v(t)\,dt  +\alpha \dual{  \lambda}{v}_{W_0^{1,1}(0,T;\R^n)} +\dual{  \lambda}{  \F'( \bar q)v}_{W_0^{1,1}(0,T;\R^n)} 
 \\=\int_0^T j'( \bar q)(t) v(t)\,dt+\dual{\bar q(T)-q_d}{v(T)}_{\R^n}
 \quad  \forall\, v \in W_0^{1,1}(0,T;\lo),
 \label{eq:adjoint1} 
 \\[1mm] \dot{\bar q}^i \xi^i=0 \quad \text{a.e.\,in }(0,T),\quad i=1,...,n,\label{xi}
 \\[1mm] \dual{\lambda^i}{\bar z^i\,v}_{W_0^{1,1}(0,T)}=0 \quad \forall\,v \in W_0^{1,1}(0,T),\quad i=1,...,n,\label{lambda}
\\[1mm] \dual{\lambda}{v}_{W_0^{1,1}(0,T;\R^n)}+(\ll,v)_{H^1(0,T;\lo)}=0\ \ \forall\, v \in H_0^1(0,T;\lo),\label{eq:grad_q}
 \end{gather} 
 \end{subequations}    where we abbreviate $\bar z:=-\alpha \bar q+\ll-\F(\bar q )$.\end{theorem} 
\begin{proof}Let $\{\ell_\e\}_{\e>0}$ be the sequence of  local minimizers of $\{ \eqref{eq:min_e}\}_{\e>0}$ from Lemma \ref{approx} and let $(q_\e,\xi_\e, \lambda_\e)$ be the tuple from Lemma \ref{thm:ss_qsep1}. Then, in light of Proposition \ref{bound}, there exists $\lambda \in W^{-1,\infty}(0,T;\R^n)$ and $
  \xi \in L^\infty(0,T;\hoon)$ so that 
   \begin{equation}\label{l_conv}
  \lambda_\e \weakly^\star \lambda \quad \text{in }W^{-1,\infty}(0,T;\R^n),
  \end{equation}
   \begin{equation}\label{xi_conv}
  \xi_\e \weakly^\star \xi \quad \text{in }L^{\infty}(0,T;\R^n).
  \end{equation}
Now, let $ v \in {W^{1,1}_0(0,T;\hoon)}$ be arbitrary but fixed. Testing \eqref{eq:adjoint1_q1} with $v$, where one uses the integration by parts formula, implies 
 \begin{equation}\label{1}
 \begin{aligned}
 \int_0^T & \xi_\e(t) \dot v(t)\,dt  +\alpha \int_0^T  \lambda_\e(t) v(t)\,dt +\int_0^T \lambda_\e(t)  \F'( q_\e)(t)v(t) \,dt 
 \\&=\int_0^T j'( q_\e)(t) v(t)\,dt+\dual{ q_\e(T)-q_d}{v(T)}_{\R^n}.
 \end{aligned}\end{equation}
 In view of Assumption \ref{assu}, we have for all $y \in L^\infty(0,T;\R^n)$
 $$\F'( y)v=\kappa'(\HH(y))(\HH'(y)v)=\kappa'(\HH(y))\HH(v) \in H_0^{1}(0,T;\R^n),$$
with
$$\frac{d}{dt} [\F'( y)v]=\kappa''(\HH(y))y(\HH(v))+\kappa'(\HH(y))v.$$
Therefore,  
\begin{equation}\label{conv_k''}
\begin{aligned}
&  \|\F'( q_\e)v- \F'( q)v\|_{W_0^{1,1}(0,T;\R^n)} 
  \\&\qquad \qquad \quad \leq \|\kappa''(\HH(q_\e))q_\e(\HH(v))- \kappa''(\HH(\bar q))\bar q(\HH(v))\|_{L^{1}(0,T;\R^n)} 
  \\&\qquad \qquad \quad  \quad +\|\kappa'(\HH(q_\e))v-\kappa'(\HH(\bar q))v\|_{L^{1}(0,T;\R^n)}
  \\& \qquad \qquad \quad  \qquad \to 0 \quad \text{as }\e \searrow 0,
\end{aligned}\end{equation}
thanks to Assumption \ref{k''} and \eqref{y_conv1}.
Hence, passing to the limit $\e \searrow 0$ in \eqref{1}, where one uses \eqref{xi_conv}, \eqref{l_conv}, \eqref{conv_k''}, \eqref{y_conv1} and Assumption \ref{assu:j}, gives in turn
\eqref{eq:adjoint1}.
Further, letting $\e \searrow 0$ in \eqref{eq:grad_q1}, where one relies on \eqref{l_conv1}, leads to  \eqref{eq:grad_q}.  
 
 Next we want to prove \eqref{xi}-\eqref{lambda}. To this end, let $i=1,...,n$ and $v \in W^{1,1}_0(0,T)$ be arbitrary but fixed.
On account of Lemma \ref{lem:ode1} and \eqref{eq:signcond_q} we have
 \begin{equation}\label{dotq_xi}\begin{aligned}
 \int_0^T   \dot q^i_\e \xi^i_\e  v \,dt
&= \int_0^T  \frac{\max\{0,z^i_\e\}}{\e}  \xi^i_\e  v \,dt
\\&= \int_0^T  {\max\{0,z^i_\e\}} \lambda^i_\e  v \,dt
\\&= \int_0^T  z^i_\e \lambda^i_\e  v \,dt.
\end{aligned} \end{equation} 
Thanks to Lemma \ref{approx} and Assumption \ref{assu} it holds 
\begin{equation}\label{z_e}
z_\e:=-\alpha  q_\e+\ell_\e-\F( q_\e )   \to \bar z \quad \text{in }H^1(0,T;\R^n),
\end{equation}
whence
$${\max\{0,z^i_\e\}}   v \to {\max\{0,\bar z^i\}}   v \quad \text{in } W_0^{1,1}(0,T)$$follows. 
Note that here we used the global Lipschitz continuity of $\kappa$ and $\max$, which implies the continuity of the operators $\kappa,\max:H^1(0,T) \to H^1(0,T)$, see {\cite[Thm.\,1]{mm1}}. Passing to the limit $\e \searrow 0$ in \eqref{dotq_xi}, where one relies on \eqref{y_conv1}, \eqref{xi_conv} and \eqref{l_conv} then results in
  \begin{equation}  
 \int_0^T  {\dot{\bar q}}^i \xi^i  v\,dt =  \dual{ \lambda^i}{{\max\{0,\bar z^i\}}  v}_{W_0^{1,1}(0,T)}=\dual{\lambda^i}{\bar z^i\,v}_{W_0^{1,1}(0,T)}  \quad \forall\,v \in W_0^{1,1}(0,T).
 \end{equation}
  Since $\bar z^i \leq 0,$ see Lemma \ref{prop:ris}, we can now deduce \eqref{lambda}, and the fundamental lemma of calculus of variations then yields \eqref{xi}. 
\end{proof}
  \begin{remark}
In light of   \eqref{khb}, Assumption \ref{assu:j} and \eqref{eq:adjoint1_q1},  the only aspect that prevents us from showing uniform bounds for $\xi_\e$ in $H^1(0,T;\R^n)$ is the lack of the estimate   \begin{equation} \label{ll2}
   \|\lambda_\e\|_{L^2(0,T;\R^n)} \leq c \quad \forall\,\e>0, \end{equation}where $c>0$ is independent of the viscous parameter $\e.$
   If \eqref{ll2} would be true, then arguing as at the beginning of the proof of Theorem \ref{thm:os} would allow us to show that $\xi \in H^1(0,T;\R^n)$ and $\lambda \in {L^2(0,T;\R^n)}.$
   However, \eqref{ll2} does not seem to be available, as it requires that the uniform Lipschitz continuity of $S_\e$ in Lemma \ref{lem:lip_c} holds   from $L^2(0,T;\R^n)$ to $L^2(0,T;\R^n)$, see the first lines of the proof of Proposition \ref{bound}. 
   
   When it comes to the derivation of optimality systems for the control of rate independent evolutions, the low regularity of the involved multipliers is not surprising even in finite dimensions. Cf.\,\cite{bk} for a similar situation.
   \end{remark}

If $\kappa$ depends linearly on the history operator and if we  aim at achieving a desired state only at the end of the process,  we can make additional statements about the sign of the adjoint state and of the multiplier. These lead to a C-stationary like limit optimality system, see Remark \ref{rem:st} below.

\begin{proposition}\label{prop:aoc}
 Let $\bar \ell \in \hhyn$ be locally optimal for \eqref{eq:oc} with associated state  $
  \bar q \in \hhyy.$  If the mappings $\kappa$ and $j$ are affine transformations, then there exists an adjoint state
$
  \xi \in L^\infty(0,T;\hoon)$ and a unique multiplier $\lambda \in W^{-1,\infty}(0,T;\R^n)$ that satisfy \eqref{eq:opt_syst} and 
 \begin{subequations}
  \begin{gather}
 \bar q^i(T)-q_d^i \geq 0 \Longrightarrow  \xi^i(t)\in [0, \bar q^i(T)-q_d^i] \quad\text{f.a.a.\,}\,t \in (0,T),\label{xiq0}
 \\ \bar q^i(T)-q_d^i \leq 0 \Longrightarrow  \xi^i(t)\in [ \bar q^i(T)-q_d^i,0] \quad\text{f.a.a.\,}\,t \in (0,T),\label{lq0}
  \\ \bar q^i(T)-q_d^i > 0 \Longrightarrow \dual{\lambda^i}{v}_{W_0^{1,1}(0,T)}\geq 0 \quad \forall\,v \in W_0^{1,1}(0,T), v\geq 0, \label{xiq}
 \\\bar q^i(T)-q_d^i < 0 \Longrightarrow \dual{\lambda^i}{v}_{W_0^{1,1}(0,T)}\leq 0 \quad \forall\,v \in W_0^{1,1}(0,T), v\geq 0,\label{lq}
 \end{gather}
  \end{subequations}for all $i=1,...,n.$ In particular,
  \begin{equation}\label{lxi}
  \xi^i(t)\dual{\lambda^i}{v}_{W_0^{1,1}(0,T)} \geq 0 \quad \text{f.a.a.\,}\,t \in (0,T),\ \forall\,v \in W_0^{1,1}(0,T), v\geq 0,
  \end{equation}for all $i=1,...,n$ for which $\bar q^i(T)\neq q_d^i$.  Moreover, 
 \[\bar q^i(T)-q_d^i = 0 \Longrightarrow  \xi^i(t)=0=\lambda^i(t) \quad \text{f.a.a.\,}t \in (0,T),\]
for all $i=1,...,n$.
    \end{proposition} 

 \begin{proof}
%
Let $\{\ell_\e\}_{\e>0}$ be the sequence of  local minimizers of $\{ \eqref{eq:min_e}\}_{\e>0}$ from Lemma \ref{approx} and let $(q_\e,\xi_\e, \lambda_\e)$ be the tuple from Lemma \ref{thm:ss_qsep1}. From Theorem \ref{thm:os} we already know that there exists $\xi$ and $\lambda$ such that \eqref{eq:opt_syst} is fulfilled. 
In the following, $i \in \{1,...,n\}$ is arbitrary but fixed and we keep in mind that $\kappa=0$ and $j=0$. 
 Testing the $i$-th component of \eqref{eq:adjoint1_q1} with  $\max\{\xi^i_\e,0\}$ leads to
\[
 \int_t^T -\dot \xi^i_\e \max\{\xi^i_\e,0\} \,ds  +\alpha \int_t^T   \lambda^i_\e \max\{\xi^i_\e,0\} \,ds =0 \quad \forall\,t \in [0,T].
  \]In light of \cite[Lem.\,3.3]{d_w}, this implies 
  $$\frac{1}{2}[\max\{\xi^i_\e(t),0\}^2-\max\{\xi^i_\e(T),0\}^2]=-\alpha  \int_t^T  {\lambda^i_\e \max\{\xi^i_\e,0\}}\,ds \leq 0 \quad \forall\,t \in [0,T], $$
 where we relied on \eqref{eq:signcond_q}. Thus,
 \begin{equation}\label{1b}
 \begin{aligned}
 \max\{\xi^i_\e(t),0\}^2 \leq \max\{\xi^i_\e(T),0\}^2=\max\{ q_\e^i(T)-q_d^i,0\}^2 \quad \forall\,t \in [0,T].
 \end{aligned}\end{equation}
 Using e.g.\,{\cite[Thm.\,3.23]{dac08}} and the convergences \eqref{xi_conv} and \eqref{y_conv1} then yields
  \begin{equation}\label{dac}
 \begin{aligned}
\int_0^T \max\{\xi^i(t),0\}^2 v(t) \,dt &\leq \liminf_{\e \to 0} \int_0^T\max\{\xi^i_\e(t),0\}^2 v(t) \,dt 
\\&\leq \lim_{\e \to 0}  \max\{ q_\e^i(T)-q_d^i,0\}^2  \int_0^T  v(t)\,dt
\\&=\max\{\bar q^i(T)-q_d^i,0\}^2  \int_0^T  v(t)\,dt \quad \forall\,v \in C_c^\infty[0,T], v\geq 0, 
 \end{aligned}\end{equation}whence
 \begin{equation}\label{11b}
 \begin{aligned}
 \max\{\xi^i(t),0\}  \leq \max\{\bar q^i(T)-q_d^i,0\} \quad \text{f.a.a.\,}t \in (0,T).
 \end{aligned}\end{equation}  Hence,
 \begin{equation}\label{imp1}
 \begin{aligned}
\bar q^i(T)-q_d^i \leq 0 \Longrightarrow  \xi^i(t)\leq 0   \quad \text{f.a.a.\,}t \in (0,T).
 \end{aligned}\end{equation}

Further, testing the $i-$th component of  \eqref{eq:adjoint1_q1} with $\min\{\xi^i_\e,0\}$ leads to
\[
 \int_t^T -\dot \xi^i_\e \min\{\xi^i_\e,0\} \,ds  +\alpha \int_t^T  \lambda^i_\e \min\{\xi^i_\e,0\} \,ds =0 \quad \forall\,t \in [0,T].
  \]
  This means that
  $$\frac{1}{2}[\min\{\xi^i_\e(t),0\}^2-\min\{\xi^i_\e(T),0\}^2]=-\alpha  \int_t^T  {\lambda^i_\e \min\{\xi^i_\e,0\}}\,ds  \leq 0 \quad \forall\,t \in [0,T],
$$cf.\,\eqref{eq:signcond_q}. From here  we follow
 \begin{equation}\label{1c}
 \begin{aligned}
 \min\{\xi^i_\e(t),0\}^2 \leq \min\{\xi^i_\e(T),0\}^2 \quad \forall\,t \in [0,T].
 \end{aligned}\end{equation}
 By arguing as in the proof of \eqref{dac} we get  $ \min\{\xi^i(t),0\}^2\leq  \min\{\bar q^i(T)-q_d^i,0\}^2$ a.e.\,in $(0,T),$ whence
 \begin{equation}\label{11c}
 \begin{aligned}
 \min\{\xi^i(t),0\}  \geq \min\{\bar q^i(T)-q_d^i,0\} \quad \text{f.a.a.\,}\,t \in (0,T).
 \end{aligned}\end{equation}
 As a consequence, it holds 
  \begin{equation}\label{imp2}
 \begin{aligned}
\bar q^i(T)-q_d^i \geq 0 \Longrightarrow  \xi^i(t)\geq 0   \quad \text{f.a.a.\,}t \in (0,T).
 \end{aligned}\end{equation}  
 Now, from \eqref{imp1} and \eqref{imp2} in combination with \eqref{11b} and \eqref{11c} we can conclude \eqref{xiq0} and \eqref{lq0}.  

To prove \eqref{xiq} and \eqref{lq}, we turn our attention to  \eqref{eq:signcond_q}, from which we first deduce 
$$0 \leq \lambda^i_\e (t)\max\{ q_\e^i(T)-q_d^i,0\} \quad \text{a.e.\,where }z^i_\e \leq 0.$$
Moreover, \eqref{1c} implies that 
$$
\bar q_\e^i(T)-q_d^i \geq 0 \Longrightarrow  \xi_e^i(t)\geq 0   \quad \forall\,t \in [0,T],
$$so that, by making use again of \eqref{eq:signcond_q}, we get
$$0 \leq \lambda^i_\e (t)\max\{ q_\e^i(T)-q_d^i,0\} \quad \text{a.e.\,where }z^i_\e > 0.$$
Hence,
  \begin{equation}\label{dac1}
 \begin{aligned}
0 &\leq \max\{ q_\e^i(T)-q_d^i,0\} \int_0^T \lambda^i_\e (t) v(t) \,dt 
\quad \forall\,v \in W_0^{1,1}(0,T), v\geq 0.
 \end{aligned}\end{equation}
In order to show \eqref{dac2} below, we argue in a similar way. From \eqref{eq:signcond_q} we have
$$ \lambda^i_\e (t)=0  \quad \text{a.e.\,where }z^i_\e <0.$$ Further, \eqref{1b} leads to the implication
$$
\bar q_\e^i(T)-q_d^i \leq 0 \Longrightarrow  \xi_e^i(t)\leq 0   \quad \forall\,t \in [0,T],
$$so that, by  \eqref{eq:signcond_q}, we infer
$$\bar q_\e^i(T)-q_d^i \leq 0 \Longrightarrow \lambda^i_\e (t) \leq 0 \quad \text{a.e.\,where }z^i_\e \geq 0,$$
whence
  \begin{equation}\label{dac2}
 \begin{aligned}
0 \leq \min\{ q_\e^i(T)-q_d^i,0\} \int_0^T \lambda^i_\e (t) v(t) \,dt 
\quad \forall\,v \in W_0^{1,1}(0,T), v\geq 0.
 \end{aligned}\end{equation}follows.
From \eqref{dac1} and \eqref{dac2} we  conclude
\begin{equation}
 \begin{aligned}
0 &\leq [ q_\e^i(T)-q_d^i] \int_0^T \lambda^i_\e (t) v(t) \,dt 
\\& \quad \overset{\e \to 0}{\longrightarrow} [\bar q^i(T)-q_d^i] \dual{\lambda^i}{v}_{W_0^{1,1}(0,T)} \quad \forall\,v \in W_0^{1,1}(0,T), v\geq 0,
 \end{aligned}\end{equation}
based on the convergences \eqref{l_conv} and \eqref{y_conv1}. This proves \eqref{xiq} and \eqref{lq}. Finally, the last assertion is due to \eqref{xiq0} or  \eqref{lq0}, which imply $\xi^i=0$, in combination with \eqref{eq:adjoint1} and $\kappa'=j'=0,$ which lead to $\lambda^i=0.$
The proof is now complete.
\end{proof}
\begin{remark}\label{rem:ft}
It would be desirable to obtain the result in Proposition \ref{prop:aoc} in the absence of the requirement that $\kappa$ is an affine function. The reason why we failed to do so is the presence of the integral operator $\HH$. If this would be replaced by an identity operator, meaning that the state dependence of the dissipation potential $\RR$ happens through a Nemytskii operator, then the result in Proposition \ref{prop:aoc} would stay true under  the condition $L_k<\alpha$. In this case, $\kappa$ does not need to be affine. Let us underline that the aforementioned smallness assumption is needed anyway  to prove existence of solutions for systems of the type
\[-\partial_q \EE(t,q(t)) \in \partial_{\dot q} \RR (q(t),\dot{q}(t)) \quad \ae (0,T).\]
This fact has been established in \cite{mie_ros}, see also \cite[Rem.\,3.3]{ris}.
\end{remark}

\section{Comparison with the complete optimality conditions}\label{4}
In this section, we want to  get an idea about the completeness of the optimality conditions established in Theorem \ref{thm:os} and Proposition \ref{prop:aoc}. To this end,  we resort to a formal Lagrange approach. For simplicity, we assume in all what follows  that $n=1$. We recall that, according to Lemma \ref{prop:ris}, \eqref{eq:n} can be rewritten as  
  \begin{equation}  
(\dot q (t),-\alpha q(t) +\ell(t)-\F(q)(t)) \in \graph \partial \II_\CC \quad  \ae (0,T).\end{equation}
Note that $$\graph \partial \II_\CC=(\{0\} \times (-\infty,0]) \cup ((0,\infty) \times \{0\})=:\MM,$$
since $\II_\CC(y)=0,$ if $y \in \CC$ and $\II_\CC(y)=\infty,$ otherwise. 
We define the Lagrangian as 
$$\LL(q,\ell,\xi,\lambda):=J(q,\ell)-(\xi,\dot q)_{L^2(0,T)}+(\lambda,-\alpha q +\ell-\F(q))_{L^2(0,T)},$$
where we abbreviate
$J(q,\ell):=j(q)+\frac{1}{2}|q(T)-q_d|^2+\frac{1}{2}\|\ell\|^2_{H^1(0,T)}.$
Then, a formal derivation of optimality conditions yields 
\begin{subequations}
\begin{gather}
\partial_q \LL(\bar q,\ll,\xi,\lambda)=0,\label{l1}
\\\partial_\ell \LL(\bar q,\ll,\xi,\lambda)=0,\label{l2}
\\(-\xi(t),\lambda(t)) \in \TT_{\MM}(\dot {\bar q}(t), \bar z(t))^\circ \quad \text{a.e.\,in }(0,T),\label{nc}
\end{gather}
\end{subequations}where  $\TT_{\MM}(\dot {\bar q}(t), \bar z(t))^\circ$ denotes the polar of the tangent cone  and $\bar z:=-\alpha \bar q +\ll-\F(\bar q)$. 
The identities  \eqref{l1}-\eqref{l2} are equivalent to 
\begin{equation}
\begin{aligned}
j'(\bar q)+\bar q(T)-q_d+\dot \xi -\alpha \lambda -[ \F'( \bar q)]^\star \lambda&=0,
\\-\ddot \ll+\ll+\lambda&=0,
\end{aligned}
\end{equation}and thus, 
they  correspond to \eqref{eq:adjoint1} and \eqref{eq:grad_q}.
Now, let us take a closer look at \eqref{nc}. {This can be rewritten as}
\begin{subequations}
\begin{gather}
\dot {\bar q}(t)>0, \bar  z(t)=0 \Rightarrow \xi(t)= 0, \label{nc1}
\\\dot{\bar q}(t)=0, \bar  z(t)<0 \Rightarrow \lambda(t)= 0, \label{nc2}
\\\dot {\bar q}(t)=0, \bar  z(t)=0 \Rightarrow \xi(t)\geq 0, \lambda(t) \geq 0\label{nc3}
\end{gather}
\end{subequations}a.e.\,in $(0,T)$.
The relation \eqref{nc1} is implied by  \eqref{xi}, while \eqref{nc2} corresponds to \eqref{lambda}. To see the latter, one assumes that $\lambda$ belongs to a Lebesgue space, in which case fundamental lemma of calculus of variations gives in turn $\lambda \bar z=0$, whence \eqref{nc2}. The situation in \eqref{nc3} is a little more delicate. In Proposition \ref{prop:aoc} we managed to show that $\xi$ and $\lambda$  have the same sign and that this depends on the value of the state at the end of the process. Thus, in the case $\bar q (T) \geq q_d$, the relation  \eqref{nc3} is true, cf. Proposition \ref{prop:aoc}. 
However, it is not clear if $\bar q (T) < q_d$ implies that $\xi$ and $\lambda$ vanish.


\begin{remark}\label{rem:st}
If we look at the terminology concerning stationarity concepts for MPECs {\cite{ScheelScholtes2000}}, then the system \eqref{nc1}-\eqref{nc3} is of strong stationary type. 
As mentioned above, our optimality conditions from Theorem \ref{thm:os} and Proposition \ref{prop:aoc} fall into this class, provided that the optimal state at the end of the process is sufficiently large (i.e., $\bar q (T) \geq q_d$).

A slightly weaker stationarity concept involves  Clarke (C-)stationary conditions. According to {\cite{ScheelScholtes2000}}, these conditions say that $\xi$ and $\lambda$  have the same sign, which is precisely the case in Proposition \ref{prop:aoc}, see \eqref{lxi}.

Let us mention that there exist other various stationarity concepts that lie between C- and strong stationarity with regard to their strength, such as 
Mordukhovich (M-)stationarity. The latter feature a condition of the type
\[\lambda \geq 0,\  \xi \geq 0 \quad \text{ or }\quad  \lambda \xi=0 \quad \text{a.e.\,in }(0,T).\] In the context of our problem this would involve proving 
\begin{equation}\label{m_stat}
\bar q(T) < q_d \Rightarrow \lambda \xi=0 \quad \text{a.e.\ in }(0,T).
\end{equation}
In view of the lack of regularity of $\lambda$ and $\xi$, see Theorem \ref{thm:os}, it is however not clear how to formulate this in variational terms. 
\end{remark}

\section*{Acknowledgment}This work was supported by the DFG grant BE 7178/3-1 for the project "Optimal Control of Viscous
Fatigue Damage Models for Brittle Materials: Optimality Systems".

\bibliographystyle{plain}
\bibliography{hist_dep_EVI}

\end{document}